\documentclass[11pt, reqno, oneside]{amsart}
\usepackage{amsmath,amsfonts,amssymb,amsthm}

\usepackage{enumerate}
\usepackage{esint}
\usepackage[pagebackref,colorlinks,citecolor=blue,linkcolor=blue]{hyperref}
\usepackage[dvipsnames]{xcolor}
\usepackage{mathtools}  

\usepackage{mathtools}
\usepackage{setspace}

\usepackage{comment}

\usepackage{geometry}
\numberwithin{equation}{section}

\theoremstyle{plain}
\newtheorem{theorem}{Theorem}

\newtheorem{lemma}{Lemma}
\newtheorem{proposition}{Proposition}

\numberwithin{theorem}{section}
\numberwithin{lemma}{section}
\numberwithin{corollary}{section}
\numberwithin{proposition}{section}

\theoremstyle{definition}
\newtheorem{definition}{Definition}

\newtheorem{example}{Example}
\newtheorem{remark}{Remark}

\numberwithin{definition}{section}
\numberwithin{example}{section}
\numberwithin{remark}{section}

\newcommand{\R}{{\mathbb R}}
\newcommand{\N}{{\mathbb N}}

\newcommand{\Om}{\Omega}

\providecommand{\vint}[1]{\mathchoice
          {\mathop{\vrule width 5pt height 3 pt depth -2.5pt
                  \kern -9pt \kern 1pt\intop}\nolimits_{\kern -5pt{#1}}}
          {\mathop{\vrule width 5pt height 3 pt depth -2.6pt
                  \kern -6pt \intop}\nolimits_{\kern -3pt{#1}}}
          {\mathop{\vrule width 5pt height 3 pt depth -2.6pt
                  \kern -6pt \intop}\nolimits_{\kern -3pt{#1}}}
          {\mathop{\vrule width 5pt height 3 pt depth -2.6pt
                  \kern -6pt \intop}\nolimits_{\kern -3pt{#1}}}}

\newcommand{\eps}{\varepsilon}
\newcommand{\loc}{\mathrm{loc}}

\newcommand{\BV}{\mathrm{BV}}

\newcommand{\liploc}{\mathrm{Lip}_{\mathrm{loc}}}

\newcommand{\ch}{\text{\raise 1.3pt \hbox{$\chi$}\kern-0.2pt}}

\DeclareMathOperator{\dive}{div}

\DeclareMathOperator{\diam}{diam}

\DeclareMathOperator{\Lip}{Lip}

\DeclareMathOperator{\Var}{Var}
\DeclareMathOperator{\pV}{pV}
\DeclareMathOperator{\PV}{PV}

\DeclareMathOperator{\eV}{eV}

\begin{document}
\title[BV functions on one-dimensional spaces]{Functions of bounded variation on complete and connected one-dimensional metric spaces
%\footnote{{\bf 2010 Mathematics Subject Classification}:
%\hfill \break {\it Keywords\,}: Functions of bounded variation, one-dimensional metric spaces,
%}
}
\author{Panu Lahti}
\address{Panu Lahti,  Institut f\"ur Mathematik, Universit\"at Augsburg, Universit\"atsstr. 14, 86159 Augsburg, Germany, {\tt panu.lahti@math.uni-augsburg.de}}
\author{Xiaodan Zhou}
\address{Xiaodan Zhou, Department of Mathematical Sciences, Worcester Polytechnic Institute, Worcester, MA 01609, USA, {\tt xzhou3@wpi.edu}}


\begin{abstract}
In this paper, we study functions of bounded variation on a complete and connected metric space with finite one-dimensional Hausdorff measure. The definition of BV functions on a compact interval based on pointwise variation is extended to this general setting. We show this definition of BV functions is equivalent to the BV functions introduced by Miranda \cite{Mir}. Furthermore, we study  the necessity of conditions on the underlying space in Federer's characterization of sets of finite perimeter on metric measure spaces. In particular, our examples show that the doubling and Poincar\'e inequality conditions are essential in showing that
a set has finite perimeter if the codimension one Hausdorff measure of the measure-theoretic boundary is finite.
\end{abstract}

\maketitle

\section{Introduction}

 Functions of bounded variation, also known as BV functions, have been extensively studied and widely applied in different areas including the calculus of variations, hyperbolic conservation laws, and minimal surfaces \cite{AFP, Bressan, Gi}. In the context of metric measure spaces, the notion of functions of bounded variation is introduced by Miranda \cite{Mir} and it has attracted significant attention in recent years (e.g. \cite{A1,AmDi,KKST,Mar2,Mar1}). Motivated by the observation that various function classes including Sobolev functions and BV functions defined on the real line $\R$ have simple characterizations, in this work we focus our study on BV functions in one-dimensional metric spaces. 
Our main result gives a simple alternative definition of BV functions in a general one-dimensional space based on pointwise variation.

%As an important application, our characterization can be used for the analysis of sets of finite perimeter, whose characteristic functions by definition are BV functions.  By simple constructions, we are able to provide good one-dimensional examples to study Federer's characterization of sets of finite perimeter. Note that Federer's characterization is extended by \cite{A1, L-Fedchar} to general metric measure that satisfies the doubling condition and supports the Poincar\'e inequality. {\color{red}However, it has not been known so far whether such conditions on the underlying spaces are necessary.} Our explicit examples in this paper answer this question and show that these conditions are really essential. Let us introduce our results more precisely below. 

Let $\Omega$ denote an open set in the Euclidean space $\mathbb{R}^n$.
A function $u\in L_{\loc}^1(\Omega)$ is said to have bounded variation in $\Omega$ if 
\[
\|Du\|(\Om):=\sup\left\{\int_\Omega u\dive \varphi\, dx :\,  \varphi\in C_c^1(\Omega; \mathbb{R}^n),\, |\varphi|\le 1\right\}<\infty.
\] 
By the Riesz representation theorem, the class of functions with bounded variation in $\Omega$, denoted by $\BV(\Omega)$, is the collection of functions whose weak first partial derivatives are Radon measures. An equivalent characterization of BV functions is given as the $L^1$ limits of sequences
of smooth functions with gradients bounded in $L^1$. By replacing smooth functions
with locally Lipschitz functions and the absolute value of the gradient by a
local Lipschitz constant, Miranda \cite{Mir}
introduced functions of bounded variation on a complete doubling metric
measure space $(X, d, \mu)$ supporting a Poincar\'e inequality. Equivalent definitions of BV functions on complete and separable metric measure spaces are studied by Ambrosio and Di Marino \cite{AmDi}. They relax the locally Lipschitz functions in Miranda's definition to a more general class of functions, with the local Lipschitz constants replaced by upper gradients. We recall the definition of BV functions on general metric measure spaces using upper gradients. 

\begin{definition}\label{BV Miranda}
 Given an open set $\Om\subset X$ and a function $u$ on $\Om$,
the total variation of $u$ in $\Om$ is defined by
\[
\|Du\|(\Om):=\inf\left\{\liminf_{i\to\infty}\int_\Om g_{u_i}\,d\mu:\, u_i\to u\textrm{ in } L^1_{\loc}(\Om)\right\},
\]
where each $g_{u_i}$ is an upper gradient of $u_i$
in $\Om$. A function $u$ is said to have bounded variation on $\Omega$
if $\|Du\|(\Om)<\infty$. 
\end{definition}

On the real line $\mathbb{R}$, various function classes usually
have simpler characterizations. For example,
upon choosing a good representative, we can identify a Sobolev function $u\in W^{1,p}([a,b])$ with an absolutely continuous function with $p$-integrable derivative \cite[Theorem 1, Page 163]
{EvGa}. Functions of bounded variation on $\mathbb{R}$ can also be characterized by pointwise variation. Recall that the pointwise variation of a function
$u\colon [a,b]\to \mathbb{R}$ is defined as
\begin{equation}\label{eq:pointwise variation classical}
\PV(u,[a,b]):=\sup\left\{\sum_{k=1}^{n-1}|u(t_{k})-u({t_{k+1}})|,\ a\le t_1\le 
\ldots \le t_n\le b\right\}.
\end{equation}

If $\Omega \subset \mathbb{R}$ is open, the pointwise variation $\PV(u,\Omega)$ is defined as $\sum_I \PV(u, I)$, where the sum runs along all the closed intervals in $\Om$.
The essential variation $\eV(u, \Omega)$ is defined as
\[
\eV(u,\Om):=\inf\left\{\PV(v, \Om):\, u=v \text{ a.e. in }\Om\right\}.
\]
For $u\in L_{\loc}^1(\Omega)$, we have $\eV(u,\Omega)=\|Du\|(\Om)$
\cite[Theorem 3.27]{AFP}.

The above characterizations of function classes can be extended
to general one-dimensional metric spaces.
Let $X$ be a complete and connected metric space with finite one-dimensional Hausdorff measure $\mathcal{H}^1(X)<\infty$. In \cite{Zhou}, the notion of absolutely continuous functions is generalized and Newtonian Sobolev functions are characterized by these absolutely continuous functions. Functions of bounded variations on curves in metric measure spaces are studied by Martio \cite{Mar2, Mar1}. 
In this work, we investigate the pointwise variation characterizations of BV functions on the above
one-dimensional space. We first give the definition:

\begin{definition}\label{BV 1d}
Let $X$ be a complete connected metric measure space with $\mathcal{H}^1(X)<\infty$. For a function $v$ on $X$, we define the pointwise variation as
\[
\pV(v,X):=\sup\left\{\sum_j |v\circ\gamma_j(\ell_{j})-v\circ\gamma_j(0)|\right\},
\]
where the supremum is taken over all finite collections of
pairwise disjoint injective
arc-length parametrized curves $\gamma_j\colon[0,\ell_{j}]\to X$. Then we define
\[
\Var(u,X):=\inf\{\pV(v,X),\,v=u\textrm{ a.e. on }X\}.
\]
A function $u\colon X\to \R$ has
bounded pointwise variation
if $\Var(u,X)<\infty$.

\end{definition}
It can be shown that when $X$ is an interval,
we have
$\Var(u,X)=\eV(u,X)$.

\begin{remark}
 In the above definition, one could replace $|v\circ\gamma_j(\ell_{j})-v\circ\gamma_j(0)|$  with $\PV(v\circ\gamma_j,[0,\ell_{j}])$ for each simple curve. Lemma \ref{lem:PV and pV} shows that the two quantities are comparable.

\end{remark}

We say that a function $\widetilde{u}$ is a good representative of $u$ if
$u=\widetilde{u}$ almost everywhere and $\Var(u,X)=\pV(\widetilde{u},X)$.
We show that every function $u$ with $\Var(u,X)<\infty$ admits a good representative.
\begin{lemma}[Existence of a good representative]
Suppose that $(X,d,\mathcal H^1)$ is a complete and connected metric measure space
with $\mathcal{H}^1(X)<\infty$. If $\Var (u,X)<\infty$, then there exists a function
$\widetilde{u}$ on $X$ with $\widetilde{u}=u$ a.e. and 
\[
\pV(\widetilde{u}, X)=\Var(u,X)=\inf\{\pV(v,X):\,v=u\textrm{ a.e. on }X\}.
\]
\end{lemma}

We show that the class of BV functions given by Definition \ref{BV 1d} is equivalent to the BV functions given in Definition \ref{BV Miranda}. The main theorem is stated below: 

\begin{theorem}[Main Theorem]\label{thm:equivalence}
Suppose that $(X,d,\mathcal H^1)$ is a complete and connected metric measure space
with $\mathcal{H}^1(X)<\infty$.
Let $u$ be a function on $X$. Then the following hold:
\begin{itemize} 
\item[(1)] If $\Vert Du\Vert(X)<\infty$, then $\Var(u,X)\le \Vert Du\Vert(X)$. 
\item[(2)] Suppose there exists a constant $C_0$ such that for all $x\in X$
\begin{equation}\label{eq:blow up condition liminf}
\liminf_{r\to 0}\frac{\mathcal H^1(B(x,r))}{r}<C_0
\end{equation}
holds.
 If $\Var(u,X)<\infty$, then $\Vert Du\Vert(X)<\infty$.
\end{itemize}
\end{theorem}

\begin{remark} In particular, if $X$ is complete, connected and Ahlfors $1$-regular
with $\mathcal H^1(X)<\infty$, a function $u$ on $X$ satisfies
$\Vert Du\Vert(X)<\infty$ if and only if $\Var(u,X)<\infty$.
\end{remark}

\begin{remark} The density upper bound \eqref{eq:blow up condition liminf} turns out to be essential in this characterization. Complete and connected metric spaces $(X,d)$ with $\mathcal{H}^1(X)<\infty$ can be constructed such that a function $u$
satisfies $\Vert Du\Vert(X)=\infty$ while $\Var(u,X)<\infty$, see Example \ref{ex:space with Poincare} and Example \ref{ex:space with doubling}.
\end{remark}

The proof for the first part of the main theorem is standard and is given in Proposition \ref{prop:if direction}. The second part requires a more delicate argument. Suppose $u$ is a function with $\Var(u,X)<\infty$.  
We first use the existence of good representatives
 to show that $\Var(v,X)$ is lower semicontinuous with respect to
 convergence in $L^1(X)$. Then 
we prove the coarea inequality stated below, first for
curve-continuous functions, i.e. functions that are continuous along every curve in $X$. A sequence of curve-continuous functions $u_i$ approximating $u$ in $L^1(X)$ can be constructed such that the limit superior of $\pV(u_i, X)$ is bounded above by $C_1\Var(u,X)$, where $C_1$ is a constant. These facts imply
the following result; $\ch_E$ denotes the characteristic function of
$E\subset X$.

\begin{lemma}[Co-area Inequality]
Let $(X,d,\mathcal H^1)$ be a complete and connected metric measure space
with $\mathcal{H}^1(X)<\infty$. Suppose there exists a constant $C_0$ such that for all $x\in X$%for every $x\in X$
\[
\liminf_{r\to 0}\frac{\mathcal H^1(B(x,r))}{r}<C_0
\]
holds. Suppose $\Var(u,X)<\infty$. Then
\[
C_1\Var(u,X)\ge \int_{\R}^{*}\Var(\ch_{\{u>t\}},X)\,dt.
\]
\end{lemma}

Using also the BV coarea formula \cite[Proposition 4.2]{Mir} (see detailed statement \eqref{eq:coarea} in Section \ref{sec:Prelis}), it now suffices to consider $u=\ch_E$ for $\Var(\ch_E, X)<\infty$. Hence the proof is completed by showing that $\Vert D\ch_E\Vert(X)$ is bounded above by $C_0\Var (\ch_E,X)$.

An interesting and important aspect of the theory of BV functions lies in the analysis of sets of finite perimeter, that is, sets whose characteristic functions are BV functions. For a set $E\subset \R^n$, Federer's characterization of sets of finite perimeter \cite{Fed} states that $E$ has finite perimeter if and only if the codimension one Hausdorff measure of its measure-theoretic boundary satisfies $\mathcal H(\partial^*E)<\infty$, see Section \ref{sec:Federer} for detailed definitions.  
Let $(X,d,\mu)$ be a complete and doubling metric measure space that supports a $1$-Poincar\'e inequality and let $E\subset X$ be a measurable set. Ambrosio \cite[Theorem 5.3]{A1} shows that if $E$ has finite perimeter then $\mathcal H(\partial^*E)<\infty$.
The converse implication of Federer's characterization in the general
metric space setting is proved by the first author in \cite[Theorem 1.1]{L-Fedchar}.

It has not been known so far whether the doubling and Poincar\'e inequality conditions on the underlying space are necessary when showing that the condition
$\mathcal{H}(\partial^* E)<\infty$ implies that $E$ is of finite perimeter.
By constructing simple explicit examples of one-dimensional spaces,
we show that these two conditions are really essential.

This paper is organized in the following way: preliminaries are covered in
Section \ref{sec:Prelis} and the proof of the main theorem is presented in
Section \ref{sec:Proofs}. In Section \ref{sec:Federer}, we construct two
examples to show the necessity of the doubling condition and the
Poincar\'e inequality in Federer's characterization.

\section{Definitions and notation}\label{sec:Prelis}

Assume throughout the paper that
$(X,d,\mathcal H^1)$ is a complete and connected metric space with
$\mathcal H^1(X)<\infty$.
If a property holds outside a set of $\mathcal H^1$-measure zero, we say that
it holds almost everywhere, abbreviated a.e.
The symbol $C$ will denote a constant that only depends on the space $X$.
We say that a measure $\mu$ is doubling if there exists a constant $C$
such that $\mu(B(x,2r))\le C\mu(B(x,r))$ for all open balls
$B(x,r)$.
The space $X$ is Ahlfors $s$-regular if there is a constant $C$ such that
\[
C^{-1}r^s\le\mu(B(x,r))\le Cr^s,
\]
whenever $x\in X$ and $0<r<{\rm diam}(X)$.
 If $X$ is Ahlfors $s$-regular with respect to $\mu$, we can replace $\mu$ by the $s$-dimensional Hausdorff measure $\mathcal{H}^s$ without losing essential information \cite[Exercise 8.11]{Hei1}.
 
 A continuous mapping $\gamma\colon [a,b]\to X$ is said to be a rectifiable curve if it has finite length. A rectifiable curve always admits an arc-length parametrization (see e.g. \cite[Theorem 3.2]{Haj}).
 If $\gamma\colon [a,b]\to X$ is a rectifiable curve and
 $g\colon \gamma([a,b])\to [0, \infty]$ is a Borel function, we define
\[
\int_{\gamma} g\,ds:=\int_0^{\ell}g(\widetilde\gamma(s))\,ds,
\]
where $\widetilde\gamma\colon [0, \ell]\to X$ is the arc-length parametrization of $\gamma$.
From now on we will assume all curves to be rectifiable and
arc-length parametrized unless
otherwise specified.

\begin{definition}[Upper gradient]\label{def:upper gradient}
Let $u\colon X\to \overline{\mathbb{R}}$. We say that a Borel
function $g\colon X\to [0,\infty]$ is an upper gradient of $u$ if
\begin{equation}\label{eq:upper gradient inequality}
|u(\gamma(\ell_{\gamma}))-u(\gamma(0))|\le \int_{\gamma}g\,ds
\end{equation}
for every curve $\gamma$.
We use the conventions $\infty-\infty=\infty$ and
$(-\infty)-(-\infty)=-\infty$.
If $g\colon X\to [0,\infty]$ is a $\mu$-measurable function
and (\ref{eq:upper gradient inequality}) holds for $1$-almost every curve,
we say that $g$ is a $1$-weak upper gradient of $u$.
A property is said to hold for $1$-almost every curve
if there exists $\rho\in L^1(X)$ such that $\int_{\gamma}\rho\,ds=\infty$
for every curve $\gamma$ for which the property fails.
\end{definition}

For $1\le p<\infty$, the Newtonian Sobolev class
$N^{1,p}(X)$ consists of those $L^p$-integrable functions
on $X$ for which there exists a $p$-integrable 
upper gradient.

The notation $u_B$ stands for an integral average, that is, 
\[
u_B:=\fint_B u\,d\mu:=\frac{1}{\mu(B)}\int_B u\  d\mu.
\]
A metric measure space supporting a Poincar\' e inequality
is defined in the following way.
\begin{definition}[Space supporting Poincar\'e inequality]\label{def:poincare}
Let $1\le p<\infty$. A metric measure space $(X, d, \mu)$ is said to support a p-Poincar\' e  inequality if there exists constants $C>0$ and $ \lambda\ge1$ such that the following holds for every pair of functions
$u\colon X\to \overline{\mathbb{R}}$ and
$g\colon X\to [0,\infty]$, where $u$ is measurable
and $g$ is an upper gradient of $u$:
\[
\fint_{B(x,r)} |u-u_{B(x,r)}|\,d\mu\le Cr\left(\fint_{B(x,\lambda r)}g^p\, d\mu\right)^{\frac{1}{p}}
\]
for every ball $B(x,r)$.
\end{definition}

A metric space $X$ is quasiconvex if every two points can be joined by a curve with length comparable to the distance between these two points. If $X$ is complete, doubling and supports a $p$-Poincar\'e inequality for
$1\le p<\infty$, then $X$ is quasiconvex \cite[Proposition 4.4]{HK1}.

We recall the following generalization of the Euclidean area formula to the case of Lipschitz maps $f$ from the Euclidean space $\mathbb{R}^n$ into a metric space $X$. The proof can be found in \cite[Corollary 8]{Kb}.

\begin{theorem}[Area formula]\label{Area}
Let $f\colon \mathbb{R}^n\to X$ be Lipschitz. Then
\[
\int_{\mathbb{R}^n}g(x)J_n(mdf_x)\ dx=\int_X\sum_{x\in f^{-1}(y)}g(x)\, d\mathcal{H}^n(y)
\]
for any Borel function $g\colon \mathbb{R}^n\to [0,\infty]$, and 
\[
\int_Ag(f(x))J_n(mdf_x)\ dx=\int_Xg(y)\mathcal{H}^0(A\cap f^{-1}(y))\, d\mathcal{H}^n(y)
\]
for $A\subset\mathbb{R}^n$ measurable and any Borel function
$g\colon X\to [0,\infty]$.
\end{theorem}

We apply the above theorem to an arc-length parametrized curve. Let
$f=\gamma$ and $\gamma\colon [0,\ell]\to X$. In this case, $J_1(mdf_x)$
equals the metric derivative defined as
\[
|\dot \gamma|(t):=\lim_{h\to 0}\frac{d(\gamma(t+h), \gamma(t))}{|h|},
\] 
and $|\dot \gamma|(t)=1$ for almost every $t\in [0,\ell]$.
Let $\Gamma=\gamma([0,\ell])$ and let $g\colon X\to [0, \infty]$ be a Borel function. It follows from Theorem \ref{Area} that
\begin{equation}\label{eq:area formula 1d}
\int_0^\ell g(\gamma(s))\, ds=\int_{\Gamma} g(y)\mathcal{H}^0([0, \ell]\cap \gamma^{-1}(y))\, d\mathcal{H}^1(y).
\end{equation}

A compact and connected 1-dimensional metric space admits a nice parametrization. The proofs of the following two classical results can be found in \cite[Theorem 4.4.7, Theorem 4.4.8]{AmTi}.

\begin{theorem}[First Rectifiability Theorem]\label{2.5.1}
If $E$ is complete and $C\subset E$ is a closed connected set such that $\mathcal{H}^1(C)<\infty$, then $C$ is compact and connected by simple curves.
\end{theorem}

\begin{theorem}[Second Rectifiability Theorem]\label{2.5.2}
If $E$ is complete, $C\subset E$ is closed and connected, and $\mathcal{H}^1(C)<\infty$, then there exist countably many
arc-length parametrized simple curves $\gamma_i\colon [0,\ell_i]\to C$ such that
\[
\mathcal{H}^1\Big(C\setminus \bigcup_{i=1}^{\infty}\gamma_i([0,\ell_i])\Big)=0.
\]
\end{theorem}

Given $u\in\liploc(X)$, we define the local Lipschitz constant by
\begin{equation}\label{eq:pointwise Lipschitz constant}
\Lip u(x):=\limsup_{y\to x}\frac{|u(y)-u(x)|}{d(y,x)}.
\end{equation}

Given an open set $\Om\subset X$ and a function $u\in L^1_{\loc}(\Om)$,
we define the total variation of $u$ in $\Om$ by
\[
\|Du\|(\Om):=\inf\left\{\liminf_{i\to\infty}\int_\Om g_{u_i}\,d\mu:\, u_i\in N^{1,1}_{\loc}(\Om),\, u_i\to u\textrm{ in } L^1_{\loc}(\Om)\right\},
\]
where each $g_{i}$ is a ($1$-weak) upper gradient of $u_i$ in $\Om$.
We say that a function $u\in L^1(\Om)$ is of bounded variation, 
and denote $u\in\BV(\Om)$, if $\|Du\|(\Om)<\infty$.
A $\mu$-measurable set $E\subset X$ is said to be of finite perimeter if $\|D\ch_E\|(X)<\infty$, where $\ch_E$ is the characteristic function of $E$.

The following coarea formula is given in \cite[Proposition 4.2]{Mir}:
if $\Omega\subset X$ is an open set and $u\in L^1_{\loc}(\Omega)$, then
\begin{equation}\label{eq:coarea}
\|Du\|(\Omega)=\int_{\R}^{*}\Vert D\ch_{\{u>t\}}\Vert(\Omega)\,dt,
\end{equation}
where we abbreviate $\{u>t\}:=\{x\in \Om:\,u(x)>t\}$.
We use an upper integral since measurability is not clear, but if either side
is finite, then both sides are finite and we also have measurability.

\section{Proofs of the main results}\label{sec:Proofs}

\textbf{Standing assumptions:}
We will assume throughout this section that
$(X,d,\mathcal H^1)$ is a  complete and connected metric measure space with
$0<\mathcal{H}^1(X)<\infty$. By the First Rectifiability Theorem
\ref{2.5.1}, it follows that $X$ is compact.

\begin{comment}
We also assume that
there is a constant $C_0$ such that for all $x\in X$,
\begin{equation}\label{eq:blow up condition liminf}
\liminf_{r\to 0}\frac{\mathcal H^1(B(x,r))}{r}< C_0.
\end{equation}
\end{comment}

\subsection{Finite total variation implies finite pointwise variation}
We prove part (1) of Theorem \ref{thm:equivalence} first.

\begin{proposition}\label{prop:if direction}
Let $u$ be a function on $X$ such that $\Vert Du\Vert(X)<\infty$.
Then $\Var(u,X)\le \Vert Du\Vert(X)$.
\end{proposition}

\begin{proof}
From the definition of the total variation we find a sequence $(u_i)$
such that
$u_i\to u$ in $L^1(X)$ and
\begin{equation}\label{eq:choice of ui}
\lim_{i\to \infty}\int_X g_{i}\,d\mathcal H^1=\Vert Du\Vert(X),
\end{equation}
where each $g_i$ is an upper gradient of $u_i$.
Passing to a subsequence (not relabeled), we also have
$u_i\to u$ a.e.
By the First Rectifiability Theorem \ref{2.5.1},
for every pair of points $x,y\in X$ we find a simple curve
$\gamma\colon [0,\ell]\to X$ with $\gamma(0)=x$ and $\gamma(\ell)=y$,
and then by \eqref{eq:area formula 1d},
\[
|u_i(y)-u_i(x)|\le \int_{\gamma}g_{i}\,ds
\le \int_{X}g_{i}\,d\mathcal H^1\to \Vert Du\Vert(X)\quad\textrm{as }i\to\infty.
\]
Thus  the functions $u_i$ are uniformly bounded.
Note that the sequence of Radon measures $g_{i}\,d\mathcal H^1$
has uniformly bounded mass, and so we know that passing to a subsequence
(not relabeled) we have
$g_{i}\,d\mathcal H^1\overset{*}{\rightharpoonup}d\nu$ for some Radon measure
$\nu$ on $X$
\cite[Theorem 1.59]{AFP}. This reference also gives the lower semicontinuity
\begin{equation}\label{eq:nu is less than Du}
\nu(X)\le \lim_{i\to \infty}\int_X g_{u_i}\,d\mathcal H^1=\Vert Du\Vert(X).
\end{equation}
Moreover, for any compact set $K\subset X$ we have
\begin{equation}\label{eq:upper semicontinuity in closed sets}
\nu(K)\ge \limsup_{i\to \infty}\int_K g_{i}\,d\mathcal H^1;
\end{equation}
see \cite[Proposition 1.62]{AFP}
(and then in fact equality holds in
\eqref{eq:nu is less than Du}). 
Define $v(x):=\limsup_{i\to\infty}u_i(x)$ for every $x\in X$, so that $v=u$ $\mathcal H^1$-a.e., and $v$ is bounded since the functions $u_i$ are uniformly bounded.
Now for every simple curve $\gamma\colon [0,\ell]\to X$ we have
\begin{align*}
|v\circ\gamma(\ell)-v\circ\gamma(0)|
&\le \limsup_{i\to\infty}|u_i\circ\gamma(\ell)-u_i\circ\gamma(0)|\\
&\le \limsup_{i\to\infty}\int_{\gamma}g_{i}\,ds\\
&= \limsup_{i\to\infty}\int_{\gamma([0,\ell])}g_{i}\,d\mathcal H^1\quad\textrm{by }\eqref{eq:area formula 1d}\\
&\le \nu(\gamma([0,\ell]))\quad\textrm{by }\eqref{eq:upper semicontinuity in closed sets}.
\end{align*}
It follows that for any
finite collection of pairwise
disjoint simple curves $\gamma_j\colon [0,\ell_j]\to X$,
\[
\sum_j |v\circ\gamma_j(\ell_{j})-v\circ\gamma_j(0)|
\le \sum_j \nu(\gamma_j([0,\ell_j]))
\le \nu(X)\le \Vert Du\Vert(X)\quad\textrm{by }\eqref{eq:nu is less than Du}.
\]
It follows
that $\pV(v,X)\le \Vert Du\Vert(X)$ and so $\Var(u,X)\le \Vert Du\Vert(X)$.
\end{proof}

\subsection{
Finite pointwise variation
implies finite total variation}

The proof of part (2) of Theorem \ref{thm:equivalence} is more involved. We divide the argument into several parts. 

\subsubsection{Existence of a good representative}

We first show that every $u$ with $\Var(u,X)<\infty$ admits a good representative $\widetilde{u}$. 
As a result, $\Var(u,X)$ turns out to be
lower semicontinuous with respect to convergence in $L^1(X)$.

Note that we can define an alternative version of the pointwise
variation of a function $v$ on $X$ by
\[
\PV(v,X):=\sup\left\{\sum_j \PV(v\circ\gamma_j)\right\},
\]
where the supremum is taken over finite collections of pairwise
disjoint simple curves $\gamma_j\colon [0,\ell_j]\to X$, and we denote
$\PV(v\circ\gamma_j):=\PV(v\circ\gamma_j,[0,\ell_{j}])$; recall
\eqref{eq:pointwise variation classical}.
Then obviously $\pV(v,X)\le \PV(v,X)$. Conversely, we have the following.

\begin{lemma}\label{lem:PV and pV}
For any function $v$ on $X$, we have $\PV(v,X)\le 2\pV(v,X)$.
\end{lemma}
\begin{proof}
Consider a simple curve $\gamma$.
Take a partition $0=t_0\le t_1\le \ldots \le  t_n= \ell_{\gamma}$.
Suppose $n$ is odd (the case of even $n$ is similar).
Then the subcurves $\gamma_{|_{[t_{k},t_{k+1}]}}$, for $k=0,2,\ldots,n-1$, are disjoint,
and so are the subcurves $\gamma_{|_{[t_{k},t_{k+1}]}}$ for $k=1,3,\ldots,n-2$.
Let $\gamma^k$ be $\gamma_{|_{[t_{k},t_{k+1}]}}$ reparametrized by arc-length.
Then
\begin{align*}
&\sum_{k=0}^{n-1}|v(\gamma(t_{k}))-v(\gamma({t_{k+1}}))|\\
&\qquad = \sum_{k=0,2,\ldots,n-1}|v(\gamma^k(0))-v(\gamma^k(\ell_{\gamma^k}))|
+\sum_{k=1,3,\ldots,n-2}|v(\gamma^k(0))-v(\gamma^k(\ell_{\gamma^k}))|.
\end{align*}
Taking supremum over all partitions, we get
$\PV(v\circ\gamma,[0,\ell_{\gamma}])\le 2\pV(v,X)$.
If we consider collections of pairwise disjoint simple curves $\gamma_j$,
and if we do the above for each $\gamma_j$, we obtain that
$\PV(v,X)\le 2\pV(v,X)$.
\end{proof}

Next we show that we can find a \emph{good representative} $\widetilde{u}$
of any function $u$, with $\pV(\widetilde{u},X)=\Var(u,X)$.
In proving this we will take inspiration from Martio \cite{Mar2}. 
Given a function $v$ on $X$ and a set $D\subset X$, we define
\[
\pV_D(v,X):=\sup\left\{\sum_j |v\circ\gamma_j(\ell_{j})-v\circ\gamma_j(0)|\right\},
\]
where the supremum is taken over finite collections of pairwise disjoint simple
curves $\gamma_j\colon [0,\ell_{j}]\to X$ with endpoints $\gamma_j(0), \gamma_j(\ell_{j})\in D$.

\begin{proposition}\label{prop:extension from a set of full measure}
Let $D\subset X$ be an arbitrary set with $\mathcal{H}^1(X\setminus D)=0$.
Suppose $\pV_D(v,X)<\infty$. 
Then there exists a function $v_e$ on $X$
such that $v_e=v$ on $D$ and $\pV(v_e,X)=\pV_D(v, X)$.
\end{proposition}
\begin{proof}
If $x\in D$, define $v_e(x)=v(x)$. Fix a point $z_0\in D$. For any point $x\in X\setminus D$, by the First Rectifiability Theorem (Theorem \ref{2.5.1}),
there exists a simple curve
$\gamma_{x}\colon [0,\ell_{x}]\to X$ with $\gamma_{x}(0)=x$ and
$\gamma_{x}(\ell_{x})=z_0$.
We define
\[
v_e(x):=\lim_{t\to 0^+,\,\gamma_x(t)\in D} v\circ \gamma_x (t).
\]
The limit exists since the quantity
\[
\sup\left\{\sum_{k=1}^{n-1}|v\circ \gamma_x(t_{k})-v\circ \gamma_x({t_{k+1}})|,
\ 0\le t_1\le 
\ldots \le t_n\le \ell_x,\,\gamma_x(t_k)\in D\right\}
\]
is finite, which follows from the condition $\pV_D(v,X)<\infty$ just
as in Lemma \ref{lem:PV and pV}.
Then we show that $v_e\colon X\to \mathbb{R}$, with $v_e=v$ on $D$,
satisfies $\pV(v_e,X)=\pV_D(v,X)$.
It is clear that $\pV(v_e,X)\ge \pV_D(v,X)$.
Conversely, let $\{\gamma_j\}_{j=1}^{n}$ be an arbitrary collection
of pairwise disjoint curves. If all the endpoints
$\gamma_j(0), \gamma_j(\ell_j)\in D$, then 
$$\sum_{j=1}^n |v\circ\gamma_j(\ell_{j})-v\circ\gamma_j(0)|=\sum_{j=1}^n |v_e\circ\gamma_j(\ell_{j})-v_e\circ\gamma_j(0)|.$$
If there exists a point $p_j=\gamma_j(\ell_j)\in X\setminus D$
(or $\gamma_j(0)\in X\setminus D$, or both), then we denote the
curve connecting $z_0$ and $p_j$ in the definition
of the function value of $v_e$ at $p_j$ by $\gamma_{p_j}\colon [0, \ell_{p_j}]\to X$.
Let $\epsilon>0$ be arbitrary. We discuss two cases:
\begin{itemize}
\item[(1)] If there exists $\delta>0$ such that $\gamma_j$ intersects
with $\gamma_{p_j}$ only at $p_j$ inside $B(p_j, \delta)$, then we
define a simple curve $\widetilde{\gamma}_j\colon [0,\widetilde{\ell}_j]\to X$ by
\[
\widetilde{\gamma}_j(t) :=
  \begin{cases}
                                   \gamma_j(t) & \text{if $0\le t\le \ell_j$} \\
                                   \gamma_{p_j}(t-\ell_j) & \text{if $\ell_j\le t\le \widetilde{\ell}_j$} 
  \end{cases}
\]
where $\widetilde{\ell}_j\le \ell_j+\delta$.
By choosing  $\widetilde{\ell}_j$ sufficiently close to $\ell_j$, we have that
\[
|v\circ\widetilde{\gamma}_j(\widetilde{\ell}_j)-v_e\circ\gamma_j(\ell_j)|<\frac{\epsilon}{2n}.
\]
Likewise, if $p_j=\gamma_j(0)\in X\setminus D$, we can also extend $\gamma_j$ slightly to $\widetilde{\gamma}_j$ by attaching a small piece of $\gamma_{p_j}$ at the endpoint such that
\[
|v\circ\widetilde{\gamma}_j(0)-v_e\circ\gamma_j(0)|<\frac{\epsilon}{2n}.
\]
\item [(2)] If for every $\delta>0$ there exists
$q\in B(p_j, \delta)$ with $q\neq p_j $ such that
$q=\gamma_j(\widetilde{t}) =\gamma_{p_j}(t)$ for some $\widetilde{t},t$,
then we define $\widetilde{\gamma}_j\colon [0,\widetilde{\ell}_j]\to X$
as the restriction of $\gamma_j$ to $[0, \widetilde{t}]$, so that
\[
\begin{aligned}
|v\circ\widetilde{\gamma}_j(\widetilde{\ell}_j)-v_e\circ \gamma_j(\ell_j)|&=|v\circ{\gamma}_j(\widetilde{t})-v_e\circ \gamma_j(\ell_j)|\\
&=|v\circ\gamma_{p_j}(t)-v_e(p_j)|\\
&\le \frac{\epsilon}{2n},
\end{aligned}
\]
if we choose $t$ sufficiently close to $0$. A similar modification works for the case when $p_j=\gamma_j(0).$
\end{itemize}
Then we get a new collection of curves $\{\widetilde{\gamma}_j\}_{j=1}^{n}$
defined as above if at least one of the endpoints of $\gamma_j$ belong
to $X\setminus D$. Furthermore, since the curves
$\gamma_j$ are pairwise disjoint, we can choose
$\delta$ sufficiently small such that the curves
$\widetilde{\gamma}_j$
are pairwise disjoint. Hence, we get that
\[
\sum_{j=1}^n |v_e\circ\gamma_j(\ell_{j})-v_e\circ\gamma_j(0)|\le \sum_{j=1}^n |{v}\circ\widetilde{\gamma}_j(\widetilde{\ell}_{j})-{v}\circ\widetilde{\gamma}_j(0)|+\epsilon.
\]
This implies that $\pV(v_e,X)\le \pV_D(v,X)$, and $\pV(v_e,X)= \pV_D(v,X)$ follows.
\end{proof}

\begin{proposition}\label{prop:good representative}
Suppose $\Var (u,X)<\infty$. Then there exists a function
$\widetilde{u}$ on $X$ with $\widetilde{u}=u$ a.e. and 
\[
\pV(\widetilde{u}, X)=\Var(u,X)=\inf\{\pV(v,X):\,v=u\textrm{ a.e. on }X\}.
\]
\end{proposition}

\begin{proof}
Take a function $v=u$ a.e. with $\pV(v,X)<\infty$.
Let $u_i\colon X\to \mathbb{R}$ be a sequence such that $u_i=v$ on $D_i$
with $\mathcal{H}^1(X\setminus D_i)=0$ and $\pV(v_i, X)\to \Var(u,X)$. Let
$D_0:=\bigcap_i D_i$.
Then $u_i=v$ on $D_0$ and $\mathcal{H}^1(X\setminus D_0)=0$.
By Proposition \ref{prop:extension from a set of full measure} there exists
$\widetilde{u}\colon X\to \R$ such that $\widetilde{u}=v$ on $D_0$ and
\[
\pV(\widetilde{u}, X)=\pV_{D_0}(v,X)=\pV_{D_0}(u_i,X)
\le \pV(u_i,X)\to \Var(u,X)\quad\text{as }i\to\infty.
\]
\end{proof}

We have the following lower semicontinuity results.

\begin{proposition}\label{prop:lower semicontinuity}
	Suppose $D\subset X$ and $v_i(x)\to v(x)$ for all $x\in D$.
	Then
	\[
	\pV_D(v,X)\le \liminf_{i\to\infty}\pV_D(v_i,X).
	\]
	Next suppose $u_i\to u$ in $L^1(X)$. Then
	\[
	\Var(u,X)\le \liminf_{i\to\infty}\Var(u_i,X).
	\]
\end{proposition}
\begin{proof}
	The first claim is easy to check.
	To prove the second, we can assume that
	the right-hand side is finite and in fact that
	$\Var(u_i,X)<\infty$ for each $i\in\N$, and then
	we can 
	choose good representatives $\widetilde{u_i}$.
	Passing to a subsequence (not relabeled)
	we have $\widetilde{u_i}(x)\to u(x)$ for every $x\in D$ with
	$\mathcal H^1(X\setminus D)=0$.
	By the first claim,
	\begin{equation}\label{eq:lsc in D}
	\begin{aligned}
	\pV_D(u, X)&\le \liminf_{i\to\infty}\pV_D(\widetilde{u_i},X)\\
	&\le \liminf_{i\to\infty}\pV(\widetilde{u_i},X)\\
	&=\liminf_{i\to\infty}\Var(u_i,X)<\infty.
	\end{aligned}
	\end{equation}
	By Proposition \ref{prop:extension from a set of full measure}, there exists an extension $u_e$ for $u$ restricted to $D$ satisfying $u_e=u$ on $D$
	and $\pV(u_e, X)=\pV_D(u, X)$. In particular, $u_e=u$ a.e. on $X$.
	We get
	\begin{align*}
	\Var(u,X)
	&={\inf\{\pV(v,X):\, v=u \text{ a.e. on } X\}}\\
	&\le \pV(u_e, X)\\
	&= \pV_D(u, X)\\
	&=\liminf_{i\to\infty}\Var(u_i,X)
	\end{align*}
by \eqref{eq:lsc in D}.
\end{proof}

\subsubsection{Approximation by curve-continuous functions}
We say that a function $v$ on $X$ is curve-continuous if
	$v\circ\gamma$ is continuous for every curve $\gamma$ in $X$. In this part, we exploit the nice properties of curve-continuous functions to show that every function with $\Var(u,X)<\infty$ is $\mathcal{H}^1$-measurable and it can be approximated
in $L^1(X)$ by a sequence of curve-continuous functions $u_i$ such that
\[
\limsup_{i\to \infty} \pV(u_i, X)\le C_1\Var(u,X)
\]
for some constant $C_1$ depending only on $C_0$ in the density upper bound condition \eqref{eq:blow up condition liminf}. We first show that every curve-continuous function is $\mathcal{H}^1$ measurable.

\begin{lemma}\label{lem:measurability}
Let $v$ be a curve-continuous function on $X$. Then $v$ is $\mathcal H^1$-measurable.
\end{lemma}

\begin{proof}
Let $t\in \R$. It suffices to show that $\{v\ge t\}$ is $\mathcal H^1$-measurable.
By curve-continuity, for each curve $\gamma\colon [0,\ell]\to X$ the set
$\gamma([0,\ell])\cap \{v\ge t\}$ is compact.
By the Second Rectifiability Theorem \ref{2.5.2}, there exist curves
$\gamma_j\colon [0,\ell_j]\to X$, $j\in\N$, such that
\[
\mathcal H^1\left(X\setminus \bigcup_{j=1}^{\infty}\gamma_j([0,\ell_j])\right)=0.
\]
The set $\bigcup_{j=1}^{\infty}(\gamma_j([0,\ell_j])\cap \{v\ge t\})$ is a Borel
set and differs from $\{v\ge t\}$ only by a set of $\mathcal H^1$-measure zero.
\end{proof}

For a function $v$ on $X$ and $t\in\R$, $r>0$, we define
the truncations $v_{t}:=\min\{t,v\}$ and $v_{t,t+r}:=\max\{t,\min\{t+r,v\}\}$.

\begin{lemma}\label{lem:truncation lemma}
Let $v$ be a curve-continuous function on $X$ with $\pV(v,X)<\infty$ and let $t\in\R$, $r>0$.
Then
\[
\pV(v_{t},X)+\pV(v_{t,t+r},X)\le \pV(v_{t+r},X).
\]
\end{lemma}
\begin{proof}
Consider a curve $\gamma$ used in estimating $\pV(v_t,X)<\infty$.
Note that $v_t\equiv t$ in $\{v\ge  t\}$.
Thus, by also reversing direction if necessary,
we can assume that $\gamma(0)\in \{v< t\}$.
Suppose also $\gamma(\ell_{\gamma})\in \{v< t\}$, but $\gamma$
intersects $\{v\ge t\}$. Let $s_1,s_2$ be
the smallest and largest number, respectively,
for which $\gamma(s_1),\gamma(s_2)\in \{v\ge t\}$; these exist by the curve-continuity.
If $\eps>0$, by curve-continuity we find
$\widetilde{s}_1<s_1,\widetilde{s}_2>s_2$ such that
$v_t(\gamma(\widetilde{s}_1))>t-\eps$ and $v_t(\gamma(\widetilde{s}_2))>t-\eps$.
Then for the subcurves $\gamma_1:=\gamma|_{[0,\widetilde{s}_1]}$
and $\gamma_2:=\gamma|_{[\widetilde{s}_2,\ell_{\gamma}]}$
(reparametrized by arc-length) we have
\[
|v_t(\gamma_1(0))-v_t(\gamma_1(\ell_{\gamma_1}))|
\ge |v_t(\gamma(0))-t|-\eps
\]
and
\[
|v_t(\gamma_2(0))-v_t(\gamma_2(\ell_{\gamma_2}))|
\ge |v_t(\gamma(\ell_{\gamma}))-t|-\eps.
\]
Thus
\[
|v_t(\gamma_1(0))-v_t(\gamma_1(\ell_{\gamma_1}))|
+|v_t(\gamma_2(0))-v_t(\gamma_2(\ell_{\gamma_2}))|\ge 
|v_t(\gamma(0))-v_t(\gamma(\ell_{\gamma}))|-2\eps.
\]
Since $\eps>0$ was arbitrary, we conclude that in the definition of 
$\pV(v,X)$, we can replace the curve $\gamma$ by two curves that
are contained in $\{v< t\}$. Similarly, if $\gamma(0)\in \{v< t\}$
and $\gamma(\ell_{\gamma})\in \{v\ge t\}$, we can replace such
$\gamma$ by one subcurve that is in $\{v< t\}$.

Now fix $\eps>0$ and take a collection of
pairwise disjoint simple curves $\gamma_j$ contained inside
$\{v< t\}$ such that
\[
\sum_{j=1}^{N_1} |v_{t}\circ\gamma_j(\ell_{j})- v_{t}\circ\gamma_j(0)|+\eps>\pV(v_{t},X).
\]
Analogously,  we find a collection of
pairwise disjoint simple curves $\gamma_j$ contained inside
$\{v> t\}$ such that
\[
\sum_{j=N_1+1}^{N_2} |v_{t,t+r}\circ\gamma_j(\ell_{j})- v_{t,t+r}\circ\gamma_j(0)|+\eps>\pV(v_{t,t+r},X).
\]
Now the curves $\gamma_j$, $j=1,\ldots, N_2$, are pairwise disjoint, and thus
\begin{align*}
&\pV(v_{t},X)+\pV(v_{t, t+r},X) \\
&\qquad \le \sum_{j=1}^{N_1} |v_{t}\circ\gamma_j(\ell_{j})- v_{t}\circ\gamma_j(0)|+\sum_{j=N_1+1}^{N_2} |v_{t,t+r}\circ\gamma_j(\ell_{j})- v_{t,t+r}\circ\gamma_j(0)|+2\eps\\
&\qquad= \sum_{j=1}^{N_2} |v_{t+r}\circ\gamma_j(\ell_{j})- v_{t+r}\circ\gamma_j(0)|+2\eps\\
&\qquad\le \pV(v_{t+r},X)+2\eps.
\end{align*}
Letting $\eps\to 0$, we get $\pV(v_{t},X)+\pV(v_{t, t+r},X)\le \pV(v_{t+r},X)$.
\end{proof}

\begin{lemma}\label{lem:level sets}
Let $v$ be a curve-continuous function on $X$ and $t\in\R$, $r>0$. Then
\[
\pV(\ch_{\{v>t\}},X)\le \liminf_{r\to 0}\frac{1}{r} \pV(v_{t,t+r},X).
\]
\end{lemma}
\begin{proof}
Let $\gamma\colon [0,\ell]\to X$ be a simple curve.
We have for every $s\in [0,\ell]$
\[
\ch_{\{v>t\}}(\gamma(s))=\lim_{r\to 0}\frac{v_{t,t+r}(\gamma(s))-t}{r}.
\]
In fact, if $v(\gamma(s))\le t,$ then $\ch_{\{v>t\}}(\gamma(s))=0$ and $v_{t,t+r}(\gamma(s))=t$. If $v(\gamma(s))> t,$  then $\ch_{\{v>t\}}(\gamma(s))=1$ . Choose $r_0$ sufficiently small such that $v(\gamma(s))\ge t+r$ for all $r\le r_0$ and then
$v_{t,t+r}(\gamma(s))=t+r$.

Now
\begin{align*}
|\ch_{\{v>t\}}\circ \gamma (\ell)-\ch_{\{v>t\}}\circ\gamma(0)|
=\lim_{r\to 0}r^{-1}
|v_{t,t+r}\circ \gamma (\ell)-v_{t,t+r}\circ \gamma (0)|.
\end{align*}

Let $\eps>0$.
Then take a collection of
pairwise disjoint injective curves $\gamma_j$ such that
\begin{align*}
\min\{\pV(\ch_{\{v>t\}},X),\eps^{-1}\}
&\le \sum_{j=1}^N |\ch_{\{v>t\}}\circ\gamma_j(\ell_{j})-\ch_{\{v>t\}}\circ\gamma_j(0)|+\eps\\
&= \sum_{j=1}^N \lim_{r\to 0}r^{-1} |v_{t,t+r}\circ\gamma_j(\ell_{j})-v_{t,t+r}\circ\gamma_j(0)|+\eps\\
&= \lim_{r\to 0}r^{-1}\sum_{j=1}^N |v_{t,t+r}\circ\gamma_j(\ell_{j})-v_{t,t+r}\circ\gamma_j(0)|+\eps\\
&\le \liminf_{r\to 0}r^{-1}\pV(v_{t,t+r},X)+\eps.
\end{align*}

Letting $\eps\to 0$, we get the result.
\end{proof}

For any functions $v,w$ on $X$, we clearly have the subadditivity
\begin{equation}\label{eq:pV subadditivity}
\pV(v+w,X)\le \pV(v,X)+\pV(w,X).
\end{equation}

Define the \emph{inner metric} $d_{in}$ by
\[
d_{in}(x,y):=\inf\{\ell_{\gamma}:\,\gamma\textrm{ is a curve such that }\gamma(0)=x,\,\gamma(\ell_{\gamma})=y\},\quad x,y\in X.
\]
Denote a ball with respect to the inner metric by $B_{in}(x,r)$.

\begin{proposition}\label{prop:approximation by continuous functions}
Suppose there exists a constant $C_0$ such that for all $x\in X$
\[
\liminf_{r\to 0}\frac{\mathcal H^1(B(x,r))}{r}<C_0
\]
holds. Suppose $\Var(u,X)<\infty$. Then $u$ is $\mathcal H^1$-measurable and
there exists a sequence of curve-continuous functions
$u_i\to u$ in $L^1(X)$ such that
\[
\limsup_{i\to\infty}\pV(u_i,X)\le C_1\Var(u,X).
\]
for a constant $C_1$ that depends only on $C_0$.
\end{proposition}

\begin{proof}
By Proposition \ref{prop:good representative} we find
a good representative $v$ of $u$.
Note that $v$ is necessarily bounded; if it were not, we could fix
a point $x_0$ and find points $x_j$ with $|v(x_j)|\to \infty$ as $j\to\infty$,
and join $x_0$ to each $x_j$ with a curve $\gamma_j$
(by the First Rectifiability Theorem), to get
\[
\pV(v,X)\ge |v(\gamma_j(\ell_{\gamma_j}))-v(\gamma_j(0))|
=|v(x_j)-v(x_0)|\to \infty\quad\textrm{as }j\to\infty.
\]
Fix $\eps>0$.
Consider all the points where $v$ is not curve-continuous; such points are
contained in the ``jump sets'', defined for $\kappa>0$ by
\begin{equation}\label{eq:jump sets}
\begin{split}
&J_{v,\kappa}:=\{x\in X: \textrm{ for all }\delta>0\textrm{ there exist pairwise disjoint curves }
\gamma_j\subset B_{in}(x,\delta)\\
&\qquad \qquad \textrm{ such that }\sum_{j}|v(\gamma_j(\ell_{j}))-v(\gamma_j(0))|\ge \kappa\}.
\end{split}
\end{equation}
We can see that each $J_{v,\kappa}$ is finite
(else we would get $\pV(v,X)=\infty$). Let
also $J_v:=\bigcup_{\kappa>0}J_{v,\kappa}$.
For every $x\in J_{v}$, we define
the ``size of the jump''
\[
J_v(x):=\sup\{\kappa>0:\,x\in J_{v,\kappa}\}.
\]
Let $\eps>0$.
The set $J_v$ is at most countable,
and so we find an open set $W_{\eps}\supset J_v$ with $\mathcal H^1(W_{\eps})<\eps$.

Let $x_k$ be an enumeration of all the points in $J_v$, with the jumps $J_v(x_k)$
in decreasing order.
Note first that by choosing suitable short curves near the
jump points, we find that
\begin{equation}\label{eq:pV inside jump set}
\pV(v,X)\ge \sum_{k=1}^{\infty}J_v(x_k).
\end{equation}

We modify $v$ as follows.
We find $r_1>0$ such that $B_1=B_{in}(x_1,r_1)\subset W_\eps$
and, using also \eqref{eq:blow up condition liminf}
(below $\pV(v,2B_1)$ means that all the curves considered are inside
$2B_1=B_{in}(x_1,2r_1)$)
\begin{equation}\label{eq:choice of delta x}
\pV(v,2B_1) \le 2J_v(x_1)\quad\textrm{and}\quad
\frac{\mathcal H^1(2B_{1})}{r_1}<2 C_0.
\end{equation}
Choose a function $\eta_1$ that is $r_{1}^{-1}$-Lipschitz
with respect to $d_{in}$, with $\eta_{1}=1$ in $B_1$
and $\eta_{1}=0$ outside $2B_{1}$.
Define ($v_{B_{1}}$ denotes integral average)
\[
w_1:=v(1-\eta_{1})+\eta_{1}\cdot v_{B_{1}}.
\]
Note that $J_{w_1}\subset J_v\setminus \{x_1\}$ and that
\begin{equation}\label{eq:decrease in jump size}
J_{w_1}(x_k)\le J_{v}(x_k)\quad\textrm{for all }k\ge 2.
\end{equation}

Note also that $w_1=v+\eta_1(v_{B_{1}}-v)$ and
consider $\pV(\eta_1(v_{B_{1}}-v),X)$.
Let $\gamma_j$ be pairwise disjoint simple curves.
Note that $\eta_1(v_{B_{1}}-v)\neq 0$ only inside the ball $2B_{1}$.
By splitting the curves
$\gamma_j$ into subcurves if necessary, we can assume that each of them is
contained inside the ball $2B_{1}$. Then we have
\begin{align*}
&|(\eta_{1}(v_{B_{1}}-v))(\gamma_j(\ell_{j}))-(\eta_{1}(v_{B_{1}}-v))
(\gamma_j(0))|\\
&\qquad\le |\eta_{1}(\gamma_j(\ell_{j}))(v_{B_{1}}-v)(\gamma_j(\ell_{j}))-
\eta_{1}(\gamma_j(\ell_{j}))(v_{B_{1}}-v)(\gamma_j(0))|\\
&\qquad\qquad + |\eta_{1}(\gamma_j(\ell_{j}))(v_{B_{1}}-v)(\gamma_j(0))-
\eta_{1}(\gamma_j(0))(v_{B_{1}}-v)(\gamma_j(0))|\\
&\qquad\le |v(\gamma_j(\ell_{j}))-v(\gamma_j(0))|+
|\eta_{1}(\gamma_j(\ell_{j}))-\eta_{1}(\gamma_j(0))|
\sup_{2B_{1}}|v_{B_{1}}-v|\\
&\qquad\le |v(\gamma_j(\ell_{j}))-v(\gamma_j(0))|+|\eta_{1}(\gamma_j(\ell_{j}))-\eta_{1}(\gamma_j(0))|
\cdot 2J_v(x_1)\quad\textrm{by }\eqref{eq:choice of delta x}\\
&\qquad\le |v(\gamma_j(\ell_{j}))-v(\gamma_j(0))|+r_{1}^{-1}\ell_{\gamma_j}\cdot 2J_v(x_1).
\end{align*}
Thus
\begin{align*}
&\sum_{j}|(\eta_{1}(v_{B_{1}}-v))(\gamma_j(\ell_{j}))-(\eta_{1}(v_{B_{1}}-v))(\gamma_j(0))|\\
&\qquad \le  \sum_{j}|v(\gamma_j(\ell_{j}))-v(\gamma_j(0))|+r_{1}^{-1}\mathcal H^1(2B_{1})2J_v(x_1)
\le  (2+4C_0)J_v(x_1)\quad\textrm{by }\eqref{eq:choice of delta x}
\end{align*}
and so
\[
\pV(\eta_{1}(v_{B_{1}}-v),X)\le (2 +4C_0) J_v(x_1).
\]
Finally, by \eqref{eq:pV subadditivity},
\begin{equation}\label{eq:pV from jump point}
\pV(w_1,X)\le \pV(v,X)+\pV(\eta_{1}(v_{B_{1}}-v),X)
\le \pV(v,X)+(2 +4C_0)J_v(x_1).
\end{equation}

Now we can do this inductively. For each $k\in\N$,
provided that $x_{k+1}\in J_{w_k}$ (if not, we just let $w_{k+1}=w_k$)
we choose
$r_{k+1}>0$ such that
$2B_{k+1}=B_{in}(x_{k+1},2r_{k+1})\subset W_\eps$ and
\[
\pV(w_k,2B_{k+1})\le 2J_{w_k}(x_{k+1})
\quad\textrm{and}\quad
\frac{\mathcal H^1(2B_{k+1})}{r_{k+1}}< 2C_0.
\]
As above, we choose a cutoff function $\eta_{k+1}$ and then define
\[
w_{k+1}:=w_k(1-\eta_{k+1})+\eta_{k+1}\cdot (w_k)_{B_{k+1}}.
\]
We claim that for all $k\in \N$, we have
\[
\pV(w_{k},X)\le \pV(v,X)+(2+4C_0)\sum_{m=1}^{k}J_{v}(x_{m})
\]
and that
\begin{equation}\label{eq:decrease in jump size 2}
J_{w_k}(x_m)\le J_{v}(x_m)\quad\textrm{for all }m\ge k+1.
\end{equation}
We have shown these to be true for $k=1$
(recall also \eqref{eq:decrease in jump size}), and \eqref{eq:decrease in jump size 2}
 is easily seen to hold with $k$ replaced by $k+1$. Moreover,
\begin{align*}
\pV(w_{k+1},X)
&\le \pV(w_{k},X)+(2+4C_0)J_{w_k}(x_{k+1})\quad\textrm{(just as in }\eqref{eq:pV from jump point})\\
&\le \pV(v,X)+(2+4C_0)\sum_{m=1}^{k}J_{v}(x_{m})+(2+4C_0)J_{w_k}(x_{k+1})\quad\textrm{by ind. hyp.}\\
&\le \pV(v,X)+(2+4C_0)\sum_{m=1}^{k+1}J_{v}(x_{m})\quad\textrm{by }\eqref{eq:decrease in jump size 2}.
\end{align*}
Then let $w:=\lim_{k\to\infty}w_k$. Note that the convergence
is uniform, in particular pointwise, since
\[
|w_{k+1}-w_k|\le 2J_v(x_{k+1})
\]
and recalling \eqref{eq:pV inside jump set}.
Now by Proposition \ref{prop:lower semicontinuity}
and \eqref{eq:pV inside jump set},
\begin{align*}
\pV(w,X)
&\le \liminf_{k\to\infty}\pV(w_k,X)\\
&\le \pV(v,X)+(2+4C_0)\sum_{k=1}^{\infty}J_{v}(x_{k})
\le (3+4C_0)\pV(v,X).
\end{align*}
Since $w_k$ has jump discontinuities on curves with jump size at most
$J_v(x_{k+1})\to 0$ as $k\to\infty$, and since $w_k\to w$ uniformly,
we see that $w$ is curve-continuous.

Recall that $w$ also depends on $\eps>0$, with $w=v$ outside
the open set $W_{\eps}$ with $\mathcal H^1(W_{\eps})<\eps$.
Recall also that $v$ is bounded, and furthermore it is easy to check
from the construction that $\inf_X v\le w\le \sup_X v$.
Choosing $\eps=1/i$ and letting
$u_i$ be the corresponding curve-continuous function $w$, we now get
$u_i\to u$ a.e., and so $u$ is $\mathcal H^1$-measurable
by Lemma \ref{lem:measurability}, and then $u_i\to u$ in $L^1(X)$ and
\[
\limsup_{i\to\infty}\pV(u_i,X)\le (3+4C_0)\pV(v,X)=(3+4C_0)\Var(u,X).
\]
\end{proof}

\subsubsection{Coarea inequality and the conclusion}%pointwise bounded variation implies finite total variation}

In the last part, we will show a coarea inequality and prove the implication from sets with finite pointwise variation to finite total variation. First we show the following coarea inequality.

\begin{proposition}\label{prop:coarea}
Suppose there exists a constant $C_0$ such that for all $x\in X$%for every $x\in X$
\[
\liminf_{r\to 0}\frac{\mathcal H^1(B(x,r))}{r}<C_0
\]
holds. Suppose $\Var(u,X)<\infty$. Then
\[
C_1\Var(u,X)\ge \int_{\R}^{*}\Var(\ch_{\{u>t\}},X)\,dt.
\]
\end{proposition}

Note that we use an upper integral since measurability is not clear.

\begin{proof}
First assume that $u$ is curve-continuous and that
$\pV(u,X)<\infty$. Define (recall that $u_t=\min\{t,u\}$)
\[
m(t):=\pV(u_t,X),\quad t\in\R.
\]
Then $m$ is an increasing function and so
\[
\pV(u,X)\ge \int_{-\infty}^{\infty}m'(t)\,dt.
\]
Let $\eps>0$. Now by Lemma \ref{lem:truncation lemma},
\begin{align*}
m(t+r)-m(t)\ge \pV(u_{t,t+r},X).
\end{align*}
Furthermore, Lemma \ref{lem:level sets} implies that
\[
\liminf_{r\to 0}\frac{m(t+r)-m(t)}{r}
\ge\liminf_{r\to 0}\frac{\pV(u_{t,t+r},X)}{r}
\ge \pV(\ch_{\{u>t\}},X).
\]
Thus we have
\begin{equation}\label{eq:coarea ineq for continuous function}
\pV(u,X)\ge \int_{\R}^{*}\pV(\ch_{\{u>t\}},X)\,dt\ge
\int_{\R}^{*}\Var(\ch_{\{u>t\}},X)\,dt.
\end{equation}
	
	Now for a general function $u$ on $X$ with $\Var(u,X)<\infty$,
	by Proposition \ref{prop:approximation by continuous functions}
	we find a sequence of curve-continuous functions $u_i$ with
	$u_i\to u$ in $L^1(X)$ and
	\[
	\limsup_{i\to \infty}\pV(u_i,X)\le C_1\Var(u,X).
	\]
	For every $x\in X$,
	\[
	\int_{-\infty}^{\infty}|\ch_{\{u_i>t\}}(x)-\ch_{\{u>t\}}(x)|\,dt
	=\int_{\min\{u_i(x),u(x)\}}^{\max\{u_i(x),u(x)\}}\,dt
	=|u_i(x)-u(x)|.
	\]
	Hence by Fubini's theorem
	(recall the measurability statement of
	Proposition \ref{prop:approximation by continuous functions})
	\begin{align*}
	\int_{X}|u_i-u|\,d\mathcal H^1
	&= \int_{X}\int_{-\infty}^{\infty}|\ch_{\{u_i>t\}}(x)
	-\ch_{\{u>t\}}(x)|\,dt\,d\mathcal H^1(x)\\
	&= \int_{-\infty}^{\infty}\int_{X}|\ch_{\{u_i>t\}}(x)
	-\ch_{\{u>t\}}(x)|\,d\mathcal H^1(x)\,dt.
	\end{align*}
	Thus $\Vert \ch_{\{u_i>t\}}-\ch_{\{u>t\}}\Vert_{L^1(X)}\to 0$
	in $L^1(\R)$
	and so we can find a subsequence of $u_i$ (not relabeled) such that
	\[
	\Vert \ch_{\{u_i>t\}}-\ch_{\{u>t\}}\Vert_{L^1(X)}\to 0\quad\textrm{for a.e. }
	t\in\R.
	\]
	Then for such $t$, by the lower semicontinuity of
	Proposition \ref{prop:lower semicontinuity},
	\[
	\Var(\ch_{\{u>t\}},X)\le\liminf_{i\to\infty}\Var(\ch_{\{u_i>t\}},X).
	\]
	We find measurable functions $h_i\ge \ch_{\{u_i>t\}}$ on $\R$ such that
	\[
	\liminf_{i\to\infty}\int_{-\infty}^{\infty}h_i(t)\,dt
	=\liminf_{i\to\infty}\int_{\R}^*\Var(\ch_{\{u_i>t\}},X)\,dt.
	\]
	Then by Fatou's lemma
	\begin{align*}
	\int_{\R}^{*}\Var(\ch_{\{u>t\}},X)\,dt
	& \le \int_{\R}^{*}\liminf_{i\to\infty}\Var(\ch_{\{u_i>t\}},X)\,dt\\
	& \le \int_{\R}\liminf_{i\to\infty}h_i(t)\,dt\\
	& \le \liminf_{i\to\infty}\int_{\R}h_i(t)\,dt\\
	& = \liminf_{i\to\infty}\int_{\R}^{*}\Var(\ch_{\{u_i>t\}},X)\,dt\\
	& \le \liminf_{i\to\infty}\pV(u_i,X)\quad\textrm{by }\eqref{eq:coarea ineq for continuous function}\\
	&\le C_1\Var(u,X).
	\end{align*}
	
\end{proof}

Due to the above coarea inequality, it will suffice to consider characteristic
functions $u=\ch_E$ for $E\subset X$.

\begin{proposition}\label{prop:only if direction for E}
Suppose there exists a constant $C_0$ such that for all $x\in X$
\[
\liminf_{r\to 0}\frac{\mathcal H^1(B(x,r))}{r}<C_0
\]
holds. Let $E\subset X$.
Then $\Vert D\ch_E\Vert(X)\le C_0\Var (\ch_E,X)$.
\end{proposition}

\begin{proof}
We can assume that $\Var (\ch_E,X)<\infty$.
By Proposition \ref{prop:good representative} we find a good representative
$v$ of $\ch_E$, so that $\pV(v,X)=\Var(\ch_E,X)$.
Let $D:=\{x\in X:\,v(x)\in\{0,1\}\}$, so that $\mathcal H^1(X\setminus D)=0$.
By Proposition \ref{prop:extension from a set of full measure} and its proof,
we know that there is a function $v_e$ on $X$ with $v_e=v$ on $D$,
taking only the values $0,1$, with
$\pV(v_e,X)=\pV_D(v,X)\le \pV(v,X)$ and so in fact $\pV(v_e,X)=\Var(\ch_E,X)$.
In conclusion, we can take the good representative to be
$\ch_F$ for $F\subset X$, and then
$\pV(\ch_F,X)=\Var(\ch_E,X)$.

Recall the definition of the jump set from \eqref{eq:jump sets};
it is not difficult to see that now
\begin{align*}
&J_{\ch_F}=\{x\in X:\, \textrm{for all }\delta>0\,\textrm{there exists a curve } \gamma\subset B_{in}(x,\delta)\\
&\qquad\qquad\textrm{ that intersects both }
F\textrm{ and }X\setminus F\}.
\end{align*}
We call this the ``curve boundary'' $\partial^c F:=J_{\ch_F}$.
Clearly any curve intersecting both $F$
and $X\setminus F$ needs to intersect also $\partial^c F$.
Now if $\mathcal H^0(\partial^c F)=\infty$, then
we can pick arbitrarily many disjoint
curves $\gamma\colon [0,\ell]\to X$ with  $|\ch_F(\gamma(\ell))-\ch_F(\gamma(0))|=1$
and thus $\pV(\ch_F,X)=\infty$.
But since $\pV(\ch_F,X)<\infty$, actually
$\mathcal H^0(\partial^c F)<\infty$.
In other words, $\partial^c F=\{x_1,\ldots,x_N\}$ with $N\le \pV(\ch_F,X)=\Var(\ch_E,X)$.

Take a sequence $\delta_i\searrow 0$ such that the balls
$B(x_j,\delta_1)$, $j=1,\ldots,N$, are pairwise disjoint. 
Fix $i\in\N$.
By \eqref{eq:blow up condition liminf}, for each $j=1,\ldots,N$ we find
$\delta_{j,i}\in (0,\delta_i)$ such that
\begin{equation}\label{eq:choice of delta ji}
\frac{\mathcal H^1(B(x_j,\delta_{j,i}))}{\delta_{j,i}}< C_0.
\end{equation}
For each $j=1,\ldots,N$, let $\eta_{j,i}$ be a
$1/\delta_{j,i}$-Lipschitz function with
$\eta_{j,i}(x_{j})=1$ and $\eta_{j,i}=0$ outside $B(x_j,\delta_{j,i})$.
Define
\[
v_i:=\max\left\{\eta_{1,i},\ldots,\eta_{N,i}\right\}\quad\textrm{and}\quad
u_i:=\max\left\{\ch_{F},\eta_{1,i},\ldots,\eta_{N,i}\right\}.
\]
Let
\[
g_i:=\sum_{j=1}^N \frac{\ch_{B(x_j,\delta_{j,i})}}{\delta_{j,i}}.
\]
Note that since the pointwise Lipschitz constant
\eqref{eq:pointwise Lipschitz constant} is an upper gradient
\cite[Proposition 1.14]{BB}, and by
\cite[Corollary 2.21]{BB}, we know that $\ch_{B(x_j,\delta_{j,i})}/\delta_{j,i}$
is a $1$-weak upper gradient of $\eta_{j,i}$
(recall Definition \ref{def:upper gradient}).
Then $g_i$ is a $1$-weak upper gradient of $v_i$.

Then we can verify that $g_i$ is a $1$-weak upper gradient of $u_i$.
For this we need to check three cases for a curve
$\gamma\colon [0,\ell]\to X$ with end points
$\gamma(0)=x$ and $\gamma(\ell)=y$.
We can assume that the pair $(v_i,g_i)$ satisfies the upper gradient
inequality on the curve $\gamma$ as well as all of its subcurves
\cite[Lemma 1.40]{BB}.
The first case is $x,y\in F$, where
\[
|u_i(x)-u_i(y)|=0\le \int_{\gamma}g_i\,ds.
\]
The second case is $x,y\in X\setminus F$. Here
\[
|u_i(x)-u_i(y)|=|v_i(x)-v_i(y)|\le \int_{\gamma}g_i\,ds.
\]
The third case is $x\in F$ and $y\in X\setminus F$.
As mentioned before, $\gamma$ now necessarily intersects $\partial^c F$.
Thus there is some $t\in[0,\ell]$ such that 
$\gamma(t)\in\partial^c F$, and thus $\gamma(t)=x_j$
for some $j$. Note that $u_i(\gamma(0))=1$,
$u_i(\gamma(t))=v_i(\gamma(t))=1$, and $u_i(\gamma(\ell))=v_i(\gamma(\ell))$.
It follows that
\begin{align*}
|u_i(\gamma(\ell))-u_i(\gamma(0))|&\le |u_i(\gamma(\ell))-u_i(\gamma(t))|
+|u_i(\gamma(t))-u_i(\gamma(0))|
\\
&=|v_i(\gamma(\ell))-v_i(\gamma(t))|\le \int_\gamma g_i\, ds.
\end{align*}
In conclusion, $g_i$ is a $1$-weak upper gradient of $u_i$.
It is easy to see that also $u_i\to \ch_E$ in $L^1(X)$.
Now we have, using \eqref{eq:choice of delta ji},
\[
\Vert D\ch_E\Vert(X)
\le \liminf_{i\to\infty}\int_X g_i\,d\mathcal H^1
\le \liminf_{i\to\infty}\sum_{j=1}^N \frac{\mathcal H^1(B(x_j,\delta_{j,i}))}{\delta_{j,i}}
\le C_0N\le C_0 \Var(\ch_E,X).
\]
\end{proof}

\begin{proposition}\label{prop:only if direction}
Suppose there exists a constant $C_0$ such that for all $x\in X$%for every $x\in X$
\[
\liminf_{r\to 0}\frac{\mathcal H^1(B(x,r))}{r}<C_0
\]
holds. Suppose $\Var (u,X)<\infty$.
Then $\Vert Du\Vert(X)\le C\Var (u,X)$.
\end{proposition}
\begin{proof}
From $\Var (u,X)<\infty$ it follows that $u$ is essentially
bounded, and $u$ is $\mathcal H^1$-measurable by
Proposition \ref{prop:approximation by continuous functions}.
Combined with the fact that
$\mathcal{H}^1(X)<\infty$, we get $u\in L^1(X)$.
By the BV coarea formula \eqref{eq:coarea},
Proposition \ref{prop:only if direction for E}, and the coarea inequality of
Proposition \ref{prop:coarea}, it follows that
\[
\Vert Du\Vert(X)=\int_{\R}^*\Vert D\ch_{\{u>t\}}\Vert(X)\,dt
\le C_0\int_{\R}^*\Var(\ch_{\{u>t\}},X)\,dt\le C_0 C_1\Var(u,X).
\]
\end{proof}

Theorem  \ref{thm:equivalence} follows by combining Proposition \ref{prop:if direction} and Proposition \ref{prop:only if direction}.

\section{Federer's characterization of sets of finite perimeter}\label{sec:Federer}

Let us briefly consider a more general metric space $(X,d,\mu)$,
where $\mu$ is a Radon measure.
The \emph{codimension one Hausdorff measure} is defined for any set $A\subset X$ by
\[
\mathcal{H}(A):=\lim_{R\rightarrow 0}\mathcal{H}_{R}(A)
\]
with
\[
\mathcal{H}_{R}(A):=\inf\left\{ \sum_{i\in I}
\frac{\mu(B(x_{i},r_{i}))}{r_{i}}:\,A\subset\bigcup_{i\in I}B(x_{i},r_{i}),\,r_{i}\le R\right\},
\]
where $I\subset \N$ is a finite or countable index set.
Note that in an Ahlfors one-regular space, $\mathcal H$ is
comparable to $\mathcal H^0$.

Given any set $E\subset X$, the
measure-theoretic boundary $\partial^{*}E$ is the set of points
$x\in X$ for which
\[
\limsup_{r\to 0}\frac{\mu(B(x,r)\cap E)}{\mu(B(x,r))}>0\quad
\textrm{and}\quad\limsup_{r\to 0}\frac{\mu(B(x,r)\setminus E)}{\mu(B(x,r))}>0.
\]

Recall from the Introduction that if $(X,d,\mu)$ is a complete metric space
such that $\mu$ is doubling and the space supports a $1$-Poincar\'e inequality,
then the condition $\mathcal H(\partial^*E)<\infty$ for a measurable
set $E\subset X$ implies that $\Vert D\ch_E\Vert(X)<\infty$. This is the ``if''
direction of Federer's characterization of sets of finite perimeter.

	Define a space as a subset of $\R^2$ as follows.
	First define for each $j\in\N$
	\[
	A_j:=\bigcup_{k=0}^{2^j-1} I_k^j,
	\]
	where
	$$I_k^j:=\left\{\left(t\cos\left(\frac{k\pi}{2^j}\right),
	t\sin\left(\frac{k\pi}{2^j}\right)\right)\in \mathbb{R}^2:\,
	t\in [-1,1]\right\}$$
	is a line segment passing through the origin with length
	$\mathcal H^1(I_k^j)= 2$. The angle between $I_k^j$ and the positive $x$-axis is $\frac{k\pi}{2^j}$ and the angle between $I_k^j$ and $I_{k-1}^j$ is $\frac{\pi}{2^j}$.
	For any set $A\subset \R^2$ and $a>0$, we let
	\[
	aA:=\{(ax, ay):\,(x,y)\in A\}.
	\]
	Then consider $\widetilde{A}_j:=2^{-2j-1}A_j$ for each $j\in\N$.  
	Note that $\widetilde{A}_j$ is a collection of $2^j$ line segments $\widetilde{I}_k^j$ with length
	$\mathcal H^1(\widetilde{I}_k^j)= 2^{-2j}$.
	
	Define
	\begin{equation}\label{set}
	X:=\bigcup_{j=1}^{\infty}\widetilde{A}_j.
	\end{equation}
	
We first show that the doubling condition is essential in the
``if'' direction of Federer's characterization.

	\begin{example}\label{ex:space with Poincare}
	Equip the set $X$ in \eqref{set} with the geodesic metric and the measure
	$\mathcal H^1$. We have
	\[
	\mathcal H^1(X)\le \sum_{j=1}^{\infty}2^{j} \mathcal H^1(\widetilde{I}_k^j)
	= \sum_{j=1}^{\infty}2^{-j} =1.
	\]
	Clearly, the density
	upper bound condition \eqref{eq:blow up condition liminf} no longer
	holds at $0$.
	Moreover, $\mathcal H^1$ is not doubling: the doubling condition fails
	when we choose
	points $x$ close to $0$ with $0\in B(x,2r)$ and $0\notin B(x,r)$. 
	
	Now we show that this space does support a
	$1$-Poincar\'e inequality.
	First consider a ball $B(0,r)$.
	Suppose $u$ is a function on $X$ with $u(0)=0$ and let $g$ be an
	upper gradient of $u$.
	Every $x\in B(0,r)$ is connected to $0$ by a line segment $I$.
	We have
	\[
	\int_{I}g\,d\mathcal H^1\ge |u(x)-u(0)|=|u(x)|.
	\]
	Note that $B(0,r)$ consists of countably many line segments
	$\{I_j\}_{j=1}^{\infty}$ that have the origin as one end point
	(some may be half-open).
	By the above, we have 
	\[
	|u(x)|\le \int_{I_j}g\,d\mathcal H^1\quad\textrm{for every }x\in I_j.
	\]
	Thus
	\begin{align*}
	\int_{B(0,r)}|u|\,d\mathcal H^1
	=\sum_{j=1}^{\infty}\int_{I_j}|u|\,d\mathcal H^1
	&\le \sum_{j=1}^{\infty}\Big(\mathcal H^1(I_j)\int_{I_j}g\,
	d\mathcal H^1\Big)\\
	&\le r\int_{B(0,r)}g\,d\mathcal H^1\quad\textrm{since }
	\mathcal H^1(I_j)\le r\textrm{ for all }j\in\N.
	\end{align*}
	
	Now consider a general ball $B(x,r)$ and a function $u\in L^1(X)$ with
	upper gradient $g$. If $B(x,r)$ is contained in only one line segment,
	the Poincar\'e inequality obviously holds since it holds in $\R$.
	So we can assume that $0\in B(x,r)$. We can also assume that
	$\int_{B(0,2r)}g\,d\mathcal H^1<\infty$ and then $u$ is a
	bounded
	function in $B(0,2r)$. Thus we can assume that $u(0)=0$.
	Now
	\begin{align*}
	\int_{B(x,r)}|u-u_{B(x,r)}|\,d\mathcal H^1
	&\le 2\int_{B(x,r)}|u|\,d\mathcal H^1\quad\textrm{(see e.g. \cite[Lemma 4.17]{BB})}\\
	&\le 2\int_{B(0,2r)}|u|\,d\mathcal H^1\\
	&\le 4r\int_{B(0,2r)}g\,d\mathcal H^1\\
	&\le 4r\int_{B(x,3r)}g\,d\mathcal H^1.
	\end{align*}
	Thus a $1$-Poincar\'e inequality holds with $C_P=4$
	and $\lambda=3$.
	
	Next, for each $j\in\N$ choose
	\[
	I_1^j=\{(t\cos(2^{-j}\pi),t\sin(2^{-j}\pi)),\,t\in [-1,1]\}
	\]
	and then let
	\begin{equation}\label{eq:definition of E}
	E:=\bigcup_{j=1}^{\infty} \widetilde{I}_1^j=\bigcup_{j=1}^{\infty}2^{-2j-1}I_1^j.
	\end{equation}
	Consider any sequence $(u_i)\subset N^{1,1}(X)$ with $u_i\to \ch_E$
	in $L^1(X)$, with upper gradients $g_i$.
	We can also assume that $u_i\to \ch_E$ a.e. Thus for each
	$j\in \N$ we can choose a point $x_j\in \widetilde{I}_1^j$, $x_j\neq 0$ and a point $x_j'$ in $\widetilde{A}_j\setminus \widetilde{I}_1^j$ 
	such that 
	\begin{itemize}
	\item[(1)] $u_i(x_j)\to 1$ as $i\to\infty$;
	\item[(2)] $u_i(x_j')\to 0$ as $i\to \infty$;
	\item[(3)] the curves $\gamma_j$ joining $x_j'$ and $x_j$ only intersect
	at the origin.
	\end{itemize}
	  
	Now
	\[
	\int_{X}g_i\,d\mathcal H^1
	\ge \sum_{j=1}^{\infty}\int_{\gamma_j}g_i\,d\mathcal H^1
	\ge \sum_{j=1}^{\infty}|u_i(x_j')-u_i(x_j)|\to \infty\quad\textrm{as }i\to\infty.
	\]
	Hence $\Vert D\ch_E\Vert(X)=\infty$.

	It is easy to check that
	$0\notin \partial^*E$ and then in fact $\partial^*E=\emptyset$.
	This shows that the ``if'' direction of
	Federer's characterization does not
	hold without the doubling condition.
	
	On the other hand, $\pV(\ch_E,X)=1$ since only a curve intersecting
	$0$
	can give nonzero variation.
	Thus we do need condition \eqref{eq:blow up condition liminf} in Proposition
	\ref{prop:only if direction for E} and Proposition \ref{prop:only if direction}.
\end{example}

The following example shows that the Poincar\'e inequality cannot be dropped in the implication from $\mathcal H(\partial^*E)<\infty$ to $\Vert D\ch_E\Vert(X)<\infty$ either.

	\begin{example}\label{ex:space with doubling}
	Equip the set $X$ in \eqref{set} with the metric inherited
	from $\R^2$ and the measure $\mathcal H^1$.
	In this case, we will show that $\mathcal H^1$ is doubling on $X$,
	but $X$ does not support any Poincar\'e inequality since it is clearly
	not quasiconvex (recall Definition \ref{def:poincare} and the paragraph
	after it).
	Let $x\in X$.
	If $x\neq 0$, we have $2^{-2k-3}\le d(x,0)\le 2^{-2k-1}$ for some $k\in\N$.
	Suppose first that $r\le 2^{-2k-4}$.
	Recalling the notation from the previous example,
	note that $\widetilde{A}_{k}$ consists of
	$2^{k}$ line segments, which are at angles $2\pi\times 2^{-k-1}$ from each other.
	By simple geometric reasoning we see that the ball $B(x,r/2)$
	is intersected by at least
	\[
	\frac{r}{2}\times 2^{2k-1}\times  (2\pi\times 2^{-k-1})^{-1}\ge 2^{3k-4}r
	\]
	line segments belonging to $\widetilde{A}_{k}$,
	each for a length at least $r/2$ inside $B(x,r)$.
	Thus
	\[
	\mathcal H^1(B(x,r))\ge 2^{3k-5}r^2.
	\]
	
	To prove a converse estimate, suppose still that
	$2^{-2k-3}\le d(x,0)\le 2^{-2k-1}$, and
	suppose that $2^{-3k-6}\le r\le 2^{-2k-4}$.
	We have $B(x,r)\cap \widetilde{A}_j=\emptyset$ for all $j\ge k+2$.
	Note that $\widetilde{A}_{k+1}$ consists of
	$2^{k+1}$ line segments, which are at angles $2\pi\times 2^{-k-2}$ from each other.
	Thus we can see that there are at most
	\[
	4r \times 2^{2k+4} \times (2\pi\times 2^{-k-2})^{-1}\le 2^{3k+6}r
	\]
	line segments intersecting $B(x,r)$, each for a length at most $2r$.
	Thus
	\[
	\mathcal H^1(B(x,r))\le 2^{3k+7}r^2.
	\]	
	Thus in total
	\begin{equation}\label{eq:upper and lower bound for measure of ball}
	2^{3k-5}r^2\le \mathcal H^1(B(x,r))\le 2^{3k+7}r^2,
	\end{equation}
	where the first inequality holds for all $r\le 2^{-2k-4}$
	and the second for all $2^{-3k-6}\le r\le 2^{-2k-4}$.
	
	Moreover, for every $k\in\N$,
	\[
	\mathcal H^1(B(0,2^{-2k-1}))\ge 2^{-2k-1} \mathcal H^1(A_k)
	=2^{-2k-1} 2^{k+1} =2^{-k}
	\]
	and so
	\begin{equation}\label{eq:measure of ball at origin}
	2^{-k}\le \mathcal H^1(B(0,2^{-2k-1}))
	\le \sum_{j=k}^{\infty}2^{-2j-1} \mathcal H^1(A_j)
	= \sum_{j=k}^{\infty}2^{-2j-1} 2^{j+1}=2^{-k+1}.
	\end{equation}
	From these, the doubling condition for balls centered at $0$ easily follows.
	Now assume again that $x\neq 0$,	
	so that $2^{-2k-3}\le d(x,0)\le 2^{-2k-1}$ for a given $k\in\N$.
	We consider four cases:
	\begin{enumerate}[(1)]
	\item If $R<2^{-3k-4}$, then $B(x,2R)$ consists of just one line segment
	and so
	\[
	\mathcal H^1(B(x,2R))=2\mathcal H^1(B(x,R)).
	\]
	\item If $2^{-3k-4}\le R\le 2^{-2k-5}$, then by
	\eqref{eq:upper and lower bound for measure of ball},
	\[
	2^{3k-5}R^2\le \mathcal H^1(B(x,R))\quad\textrm{and}\quad
	\mathcal H^1(B(x,2R))\le 2^{3k+7}(2R)^2,
	\]
	and so we have
	\[
	\mathcal H^1(B(x,2R))\le 2^{14}\mathcal H^1(B(x,R)).
	\]
	\item If $2^{-2k-5}<R\le 2^{-2k+1}$, then
	applying \eqref{eq:upper and lower bound for measure of ball}
	with $r=2^{-2k-5}$,
	\[
	\mathcal H^1(B(x,R))
	\ge \mathcal H^1(B(x,2^{-2k-5}))
	\ge 2^{3k-5}(2^{-2k-5})^2=2^{-k-15}
	\]
	and by \eqref{eq:measure of ball at origin},
	\[
	\mathcal H^1(B(x,2R))\le \mathcal H^1(B(0,2^{-2k+2}))\le 2^{-k+3},
	\]
	and so we have
	\[
	\mathcal H^1(B(x,2R))\le 2^{18}\mathcal H^1(B(x,R)).
	\]
	\item If $2^{-2k+1}<R\le 2^{-2}$ with $k\ge 2$ (note that $\diam X=2^{-2}$),
	we choose $j\le k$ such that $2^{-2j+1}<R \le 2^{-2j+3}$. 
	Note that $B(0,R/2)\subset B(x,R)\subset B(x,2R)\subset B(0,4R)$.
	Now by \eqref{eq:measure of ball at origin},
	\[
	\mathcal H^1(B(x,R))
	\ge \mathcal H^1(B(0,R/2))
	\ge \mathcal H^1(B(0,2^{-2j}))\ge 2^{-j}
	\]
	and
	\[
	\mathcal H^1(B(x,2R))\le \mathcal H^1(B(0,4R))
	\le \mathcal H^1(B(0,2^{-2j+5}))\le 2^{-j+4}.
	\]
	Thus
	\[
	\mathcal H^1(B(x,2R))\le 2^{4}\mathcal H^1(B(x,R)).
	\]
	\end{enumerate}
	In total, the doubling condition always holds with doubling constant
	$2^{18}$, when $x\neq 0$.

	Finally, define the set $E$ as in \eqref{eq:definition of E}.
	As before, we obtain that
	$\Vert D\ch_E\Vert(X)=\infty$, $\pV(\ch_E,X)=1$, and
	$\partial^*E=\emptyset$. Thus again we see that
	the ``if'' direction of Federer's characterization does not hold,
	and that condition \eqref{eq:blow up condition liminf} is
	needed in Proposition
	\ref{prop:only if direction for E} and Proposition \ref{prop:only if direction}.
\end{example}

\end{document}